\documentclass[11pt]{article}
\usepackage[pagebackref=true,colorlinks,linkcolor=blue,citecolor=magenta]{hyperref}
\usepackage{amssymb,amsmath}
\usepackage{amsmath}
\usepackage{dsfont}
\usepackage{amsfonts}
\newtheorem{precor}{{\bf Corollary}}

\newtheorem{precon}{{\bf Conjecture}}

\newtheorem{predefin}{{\bf Definition}}

\newenvironment{defin}[1]{\begin{predefin}{\hspace{-0.5
                   em}{\bf.\ }}{\rm
#1}\hfill{$\spadesuit$}}{\end{predefin}}
\newtheorem{preexm}{{\bf Example}}

\newtheorem{prerem}{{\bf Remark}}

\newtheorem{preappl}{{\bf Application}}

\newtheorem{prelem}{{\bf Lemma}}

\newtheorem{preproof}{{\bf Proof.\ }}

\newenvironment{proof}[1]{\begin{preproof}{\rm
               #1}\hfill{$\blacksquare$}}{\end{preproof}}
\newtheorem{presproof}{{\bf Sketch of Proof.\ }}

\newtheorem{prethm}{{\bf Theorem}}

\newenvironment{thm}{\begin{prethm}{\hspace{-0.5
               em}{\bf.\ }}}{\end{prethm}}
\newtheorem{prealphthm}{{\bf Theorem}}

\newenvironment{alphthm}{\begin{prealphthm}{\hspace{-0.5
               em}{\bf.\ }}}{\end{prealphthm}}
\newtheorem{prepro}{{\bf Proposition}}

\newtheorem{preprb}{{\bf Problem}}

\def\conct[#1,#2]{\mbox {${#1} \leftrightarrow {#2}$}}
\def\dconct[#1,#2]{\mbox {${#1} \rightarrow {#2}$}}
\def\deg[#1,#2]{\mbox {$d_{_{#1}}(#2)$}}
\def\mindeg[#1]{\mbox {$\delta_{_{#1}}$}}
\def\maxdeg[#1]{\mbox {$\Delta_{_{#1}}$}}
\def\outdeg[#1,#2]{\mbox {$d_{_{#1}}^{^+}(#2)$}}
\def\minoutdeg[#1]{\mbox {$\delta_{_{#1}}^{^+}$}}
\def\maxoutdeg[#1]{\mbox {$\Delta_{_{#1}}^{^+}$}}
\def\indeg[#1,#2]{\mbox {$d_{_{#1}}^{^-}(#2)$}}
\def\minindeg[#1]{\mbox {$\delta_{_{#1}}^{^-}$}}
\def\maxindeg[#1]{\mbox {$\Delta_{_{#1}}^{^-}$}}

\def\dre[#1,#2,#3]{\mbox {${\cal E}_{_{#3}}(#1,#2)$}}
\def\pdre[#1,#2,#3]{\mbox {${\cal P}_{_{#3}}(#1,#2)$}}
\def\var[#1,#2]{\mbox {${\rm Var}_{_{#1}}(#2)$}}
\def\ls[#1]{\mbox {$\xi^{^{#1}}$}}
\def\hom[#1,#2]{\mbox {${\rm Hom}({#1},{#2})$}}
\def\onvhom[#1,#2]{\mbox {${\rm Hom^{v}}(#1,#2)$}}
\def\onehom[#1,#2]{\mbox {${\rm Hom^{e}}(#1,#2)$}}
\def\core[#1]{\mbox {$#1^{^{\bullet}}$}}
\def\cay[#1,#2]{\mbox {${\rm Cay}({#1},{#2})$}}
\def\cays[#1,#2]{\mbox {${\rm Cay_{s}}({#1},{#2})$}}
\def\dirc[#1]{\mbox {$\stackrel{\rightarrow}{C}_{_{#1}}$}}
\def\cycl[#1]{\mbox {${\bf Z}_{_{#1}}$}}

\begin{document}
\begin{center}
{\Large \bf A Note On Cover-Free Families}\\
\vspace*{0.5cm}
{\bf Mehdi Azadimotlagh}\\
{\it Department of Mathematics}\\
{\it Kharazmi University, 50 Taleghani Avenue, 15618, Tehran, Iran}\\
{\tt std$\_$m.azadim@khu.ac.ir}\\
\end{center}
\begin{abstract}
Let $ N((r, w;d),t) $ denote the minimum number of points in a $ (r, w;d)- $cover-free family having $ t  $ blocks. 
Hajiabolhassan and Moazami (2012)~\cite{Hajiabolhassan20123626} showed that the Hadamard conjecture is equivalent to confirm $N((1,1; d), 4d-1)=4d-1$. 
Hence, it is a challenging and interesting problem to determine the exact value of $N((r,w;d),t)$. 
In this paper, we determine the exact value of $N((r,w; d), t)$ for every $r$, $w$, where  
$r+w \leq t$ and some $d$.
\begin{itemize}
\item[]{{\footnotesize {\bf Key words:}\ Cover-free families, Biclique covering number}}
\item[]{ {\footnotesize {\bf Subjclass:} 05B40. }}
\end{itemize}
\end{abstract}
\section{Introduction}
A family of sets is called an $ (r , w) $-cover-free family (or $ (r , w) $-CFF) if no intersection of $ r $ sets of the family are covered by a union of any other $ w $ sets of the family. Cover-free families were first described  by Kautz and Singletonin (1964) to investigate superimposed binary codes ~\cite{1053689}. Erdos et al. ~\cite{erdos1} introduced the  $ (1,r)-$cover-free family as a generalization of Sperner family. Stinson et al. ~\cite{Stinson2000595} considered cover-free families as  group testing. Mitchell and Piper ~\cite{Mitchell1988215} considered a key distribution pattern which appears to be equivalent to the notion of cover-free family. For another application and  discussion of cover-free families, (see, for example, ~\cite{2014arXiv1410.4361A, rcff2,haji2, Hajiabolhassan20123626,rcff3, Stinson2004463, wei}). Stinson  and Wei ~\cite{Stinson2004463} have introduced a generalization of cover-free families as follows. 
\begin{defin}{ Let $d, n, t, r,$ and $ w $ be positive integers  and $ B = \{ B_1, \ldots, B_t \}$ be a collection of subsets of a set $X$, where
$|X| = n$. Each element of the collection $B$ is called a block and the elements of $X$ are called points. The pair $(X,B)$ is called an $ (r,w; d)-CFF(n, t) $ if for any two sets of indices $ L,M \subseteq [t] $ such that $ L \cap M = \emptyset , |L| = r, $ and $ |M| = w, $ we have
$$ \vert (\bigcap_{l \in L}^{} B_{l}) \setminus (\bigcup_{m \in M}^{} B_{m}) \vert \geq d .$$
Let $N((r,w;d),t)$ denote the minimum number of points of $X$ in an  $(r,w;d)-CFF$ having $t$ blocks. 
} 
\end{defin}
As was shown by Engel ~\cite{CPC:1774324}, determining the optimal value for a cover-free family is NP-hard.
Also, Hajiabolhassan and Moazami ~\cite{Hajiabolhassan20123626} showed that the existence of Hadamard matrices results from the existence of some cover-free families and vice versa. A Hadamard matrix of order $n$ is an $n \times n$ matrix $H$ with entries $+1$ and $-1$, such that $ HH^T=nI_n $. 
\begin{alphthm}\label{bc-hadamard}{\rm \cite{Hajiabolhassan20123626}}
Let $d$ be a positive integer, then $N ((1, 1; d), 4d - 1) = 4d - 1$ if and only if there exists a Hadamard matrix of order $4d$. 
\end{alphthm}
It is proved that if $ H $ is a Hadamard matrix of order $ n $, then $ n=1 $, $ n=2 $, or $ n=4d $ whenever $ d $ is a positive integer ~\cite{van2001course}. It was conjectured  by Jacques Hadamard (1893) that there exists a Hadamard matrix of every order $ 4d $ whenever $ d $ is a positive integer. Actually Hajiabolhassan and Moazami showed that the Hadamard conjecture is equivalent to confirm $N ((1, 1; d), 4d - 1) = 4d - 1$. 
Thus the problem of determining the exact value of the parameter $N((r,w; d), t)$, even for special values of $r$, $w$, $d$, and $t$ is a challenging and interesting problem. In this paper, we determine the exact value of $N((r,w; d), t)$ for every $r$, $w$, where $r+w \leq t$ and some $d$.
\section{Cover-Free Family}
In this section, we restrict our attention to determine the exact value of $N((r;w; d), t)$ for every value of $r$, $w$, and $t$, where $r+w \leq t$, and some special value of $d$. In this regard, we need to use some notation and theorem as follows. A biclique of $G$ is a complete bipartite subgraph of $G$. The $d$-biclique covering (resp. partition) number  $bc_{d}(G)$ (resp. $bp_{d}(G)$) of a graph $G$ is the minimum number of bicliques of $G$ such that every edge of $G$ belongs to at least (resp. exactly) $d$  of these bicliques. Hajiabolhassan and 
Moazami~\cite{Hajiabolhassan20123626} showed that the existence of an $(r,w;d)$-cover-free family is equivalent to  the existence of $d$-biclique cover of bi-intersection graph. The bi-intersection graph $ I_{t}(r, w) $ is  a bipartite graph whose vertices are all $w$- and $r$-subsets of a  $t$-element set, where a $w$-subset is adjacent to an $r$-subset if and only if their intersection is empty. 
\begin{alphthm}\label{haji}{\rm \cite{Hajiabolhassan20123626}}
Let $r$, $w$, $d$ and $t$, be positive integers, where $t \geq r+w$. It holds that
$N((r,w; d), t) = bc_{d}(I_t(r,w))$.
\end{alphthm}

\begin{thm}\label{bi intersect2} 
Let $r$, $w$, and $ t $  be positive integers, where $t \geq r+w $.  Also, assume that the function ${ x \choose r} { t- x \choose w } $ is maximized for  $x=t'$. If $d={ t-r-w \choose t' - r}$, then 
 $$N((r,w;d),t)=bc_{d}(I_{t}(r,w))=bp_{d}(I_{t}(r,w))={ t \choose  t'}.$$
\end{thm}
\begin{proof}{Set $t'' = { t \choose t' } $. First, we show that $ I_{t}(r,w) $ can be covered by $ t'' $ bicliques such that every edge of $ I_{t}(r,w) $ is covered by exactly $d$ bicliques. Denote the vertex set of $ I_{t}(r,w) $ by bipartition $ (X,Y ) $ in which $X={[t] \choose r} $ and $Y={[t] \choose w}$.
Suppose that $ A $ is a $ t' $-subset of $[t] $ and $ A^{c} $ is the complement of the set $A$ in  $ [t] $. 
Denote the number of  these pairs by $ t'' $. Now, for every $t'$-subset $A_{j}$ of $[t] $, where $ 1 \leq j \leq  t''$, construct the biclique $ G_{j} $ with the vertex set 
$ (X_{j} , Y_{j}) $, where $X_{j}={ A_{j}\choose r} $ and $ Y_{j}={ A_{j}^{c}\choose w} $.
Let $UV$ be an arbitrary edge of $ I_{t}(r,w)$, where $ \vert U \vert = r $ and  $ \vert V \vert =w $.
In view of the definition of  $G_j$, $UV$ is covered by every  $G_{j}$ with vertex set $(X_{j} , Y_{j})$, where $U$ is a vertex of $X_{j}$  and $V$ is a vertex  of $Y_{j}$. Thus every edge of $ I_{t}(r,w)$ is covered by at least $ d $ bicliques. One can see that 
$$\sum_{j=1}^{t''}|E(G_{j})| = {t \choose t'} {t' \choose r}{ t-t' \choose w } \quad \quad \& \quad \quad \vert E( I_{t}(r,w)) \vert = {t \choose r}{t-r \choose w}. $$
Now, it is simple to check that
$$\sum_{j=1}^{t''}|E(G_{j})| = d \vert E( I_{t}(r,w)) \vert.$$
Thus every edge of $I_{t}(r,w)$ is covered by exactly $ d $ bicliques. Note that we have actually proved that 
\begin{equation}\label{eq-bi intersect1-2}
bp_{d}(I_{t}(r,w)) \leq t''.
\end{equation}
Conversely, one can see that 
$$bp_{d}(I_{t}(r,w)) \geq bc_{d}(I_{t}(r,w)) \geq \frac{d|E(I_{t}(r,w))|}{B(I_{t}(r,w))}.$$
Also, In view of the definition of  $t'$, we have
$$\frac{d|E(I_{t}(r,w))|}{B(I_{t}(r,w))}=\frac{{ t-r-w \choose t' - r}{ t \choose r }{t-r \choose w }}{{ t' \choose r} { t- t' \choose w }}={ t\choose t' }= t''. $$
Hence, 
\begin{equation}\label{eq-bi intersect2-2}
bp_{d}(I_{t}(r,w)) \geq bc_{d}(I_{t}(r,w)) \geq t''.
\end{equation}
From ~(\ref{eq-bi intersect1-2}) and ~(\ref{eq-bi intersect2-2}), we conclude
$$bp_{d}(I_{t}(r,w)) = bc_{d}(I_{t}(r,w)) = t''.$$
By Theorem~\ref{haji},
$$N((r,w;d),t)=bc_{d}(I_{t}(r,w)),$$
 this completes the proof. 
}
\end{proof}






{\bf Acknowledgments:}  This paper is a part of Mehdi Azadi Motlagh's Ph.D. Thesis. 
The author would like to express his deepest gratitude to Professor Hossein Hajiabolhassan for his invaluable comments and discussion. 
\bibliographystyle{plain}

\end{document}